\documentclass[a4paper,reqno]{amsart}

\usepackage{amssymb, longtable}
\usepackage[mathcal]{euscript}

\usepackage{color}

\usepackage{hyperref}                  

\newcommand{\cA}{\mathcal{A}}

\DeclareMathOperator{\Der}{\mathrm{Der}}

\DeclareMathOperator{\LocDer}{\mathrm{LocDer}}

\numberwithin{equation}{section}

\newtheorem{theorem}{Theorem}[section]
\newtheorem{proposition}[theorem]{Proposition}
\newtheorem*{proposition*}{Proposition}

\newtheorem{corollary}[theorem]{Corollary}

\theoremstyle{definition}
\newtheorem{df}[theorem]{Definition}
\newtheorem{example}[theorem]{Example}

\theoremstyle{remark} \newtheorem{remark}[theorem]{Remark}

\begin{document}
	
\title[Local derivations and automorphisms of nilpotent Lie algebras]{Local derivations and automorphisms of nilpotent Lie algebras}

	\author[A.Kh. Khudoyberdiyev]{Abror Khudoyberdiyev$^{1,2}$}
	\address{$^1$V.I.Romanovskiy Institute of Mathematics Uzbekistan Academy of Sciences. 9,  University street, 100174, Tashkent, Uzbekistan}

 \address{$^2$National University of Uzbekistan named after Mirzo Ulugbek. 4, University street, 100174, Tashkent, Uzbekistan}

	\author[D.E.Jumaniyozov]{Doston Jumaniyozov$^1$}
\address{E-mail addresses: khabror@mail.ru, dostondzhuma@gmail.com}

	\keywords{Lie algebra, nilpotent algebra, derivation, local derivation, automorphism, local automorphism}

\begin{abstract}
	
The paper is devoted to the study of local derivations and automorphisms of nilpotent Lie algebras. Namely, we proved that nilpotent Lie algebras with indices of nilpotency $3$ and $4$ admit local derivation (local automorphisms) which is not a derivation (automorphisms). Further, it is presented a sufficient condition under which a nilpotent Lie algebra admits a local derivation which is not a derivation. With the same condition, it is proved the existence of pure local automorphism on a nilpotent Lie algebra. Finally, we present an $n$-dimensional non-associative algebra for which the space of local derivations coincides with the space of derivations.
\end{abstract}

\date{\today}

\maketitle

\section{Introduction}

Local mappings on non-associative algebras have been one of the main topic of many research works in recent years.
The notions that are close to derivations and automorphisms (so-called local derivations and  local automorphisms) were  introduced and investigated independently by R.V. Kadison~\cite{Kadison90} and D.R. Larson and A.R. Sourour~\cite{Larson90}.
The papers \cite{Kadison90,Larson90} gave rise to a series of works devoted to the description of local automorphisms and local derivations of $C^\ast$-algebras and  operator algebras. Namely, for the algebra $\cA$ of all bounded linear operators on a Banach space $X$ it is proved that every invertible  local automorphism of $\cA$ is an automorphism \cite{Larson90}.

The coincidence of the space of local derivations ${\rm LocDer}(-)$ with the space of derivations ${\rm Der}(-)$ on semisimple Lie algebra over an algebraically closed field of characteristic zero is established, while for some finite-dimensional nilpotent Lie algebras these spaces are different (see in \cite{AK16}). Similar result for standard Borel subalgebras of a finite-dimensional simple Lie algebra over an algebraically closed field of characteristic zero was proved in \cite{YC20}. Moreover, it is proved that every local derivation is a derivation on a complex solvable Lie algebra with nilradical of maximal rank (see \cite{KKO21}). Further, it is proved that every local derivation is a derivation on infinite-dimensional Lie algebra $W(2,2)$ (see \cite{WGL24}). In general, the behavior of local derivations is not definite on solvable Lie algebras because of the existence of solvable Lie algebras which admit local derivations which are not ordinary derivation and also algebras for which every local derivation is a derivation \cite{AKh21}. The first example of a simple (ternary) algebra with non-trivial local derivations is constructed in \cite{FKK2021}. After that, the first example of a simple (binary) algebra with non-trivial local derivations was constructed in \cite{AEK21}.

As for local automorphism, there are several results which devoted to the description of properties of local automorphisms on some classes of algebras. Namely, it was proved that every local automorphism on the special linear Lie algebra $\mathfrak{sl}_n$ is either automorphism or anti-automorphism \cite{AK18}. Further Constantini \cite{Cons} extended the above result for an arbitrary complex simple Lie algebra $\mathfrak{g}.$
Namely, his main result can be formulated as follows:
\[{\rm LocAut}(\mathfrak{g})={\rm Aut}(\mathfrak{g})\rtimes\{{\rm Id}_{\mathfrak{g}},-{\rm Id}_{\mathfrak{g}}\},\]
where ${\rm Id}_{\mathfrak{g}}$ is the identical automorphism of complex simple Lie algebra $\mathfrak{g}.$

In fact, the algebra of all square matrices over the field of complex numbers, finite-dimensional simple Leibniz algebras over an algebraically closed field of the characteristic zero, infinite-dimensional simple Witt algebra over an algebraically closed field of zero characteristic do not admit pure local automorphisms (see \cite{AKO19}, \cite{BS93}, \cite{CZZ2020}). On the other hand, some types of finite-dimensional nilpotent Lie and Leibniz algebras, as well as the algebra of lower triangular $n\times n$-matrices admit local automorphisms which are not automorphisms (see \cite{AK18}, \cite{AKO19}, \cite{Elisova13}). Moreover, in \cite{KKO23}, it is proved that any local automorphism is an automorphism on complex solvable Lie algebra of maximal rank.

It should be noted that the space $\LocDer(\cA)$ coincides with the space $\Der(\cA)$ mainly when $\cA$ has no outer derivations (see \cite{AK16, CZZ21}). But, it is known that every nilpotent Lie algebra has an outer derivation. Moreover, the fact that every linear transformation on a Lie algebra of nilindex $2$ is a derivation yields that the spaces $\LocDer(\cA)$ and $\Der(\cA)$ coincide.  So, it is natural to conjecture that all nilpotent Lie algebras with nilindex greater than $3$ admit pure local derivations. It is one of the open problems regarding local mappings on non-associative algebras. In this paper, authors attack this problem on nilpotent Lie algebras with small indices of nilpotency. Namely, it is proved that nilpotent Lie algebras with nilindex $3$ and $4$ admit pure local derivations. Moreover, we presented a sufficient condition under which every nilpotent Lie algebra admits local derivation which is not derivation. In addition, we address the existence of pure local automorphism on nilpotent Lie algebras. Interestingly, the situations of local derivations and automorphisms are similar on nilpotent Lie algebras, that is under the same conditions a nilpotent Lie algebra admits pure local automorphism and derivation.

\section{General definitions}

\begin{df}
A Lie algebra over the field $\mathbb{F}$ is a vector space $\mathfrak{g}$ endowed with bilinear mapping $[-,-]:\mathfrak{g}\times\mathfrak{g}\to\mathfrak{g},$ satisfying the following identities:
\[[x,x]=0,\]
\[[[x,y],z]+[[y,z],x]+[[z,x],y]=0.\]
\end{df}

The center of a Lie algebra $\mathfrak{g}$ is defined as
$$Z(\mathfrak{g})=\{x\in\mathfrak{g} \ | \ [x,y]=0,  \mbox{ for all } y\in\mathfrak{g}\}.$$

For a Lie algebra $\mathfrak{g}$, consider the following sequence:
\[\mathfrak{g}^{1}=\mathfrak{g}, \quad \mathfrak{g}^{k+1}=[\mathfrak{g},\mathfrak{g}^{k}], \quad k\geq1.\]

If $\mathfrak{g}^{p}=0$ for some $p\in\mathbb{N},$ then $\mathfrak{g}$ is said to be {\it nilpotent}. The minimal number with this property is called an {\it index of nilpotency} (shortly {\it nilindex}) of $\mathfrak{n}.$ If the index of nilpotency of $\mathfrak{g}$ is $3,$ then $\mathfrak{n}$ is called $2$-step nilpotent Lie algebra.

Note that the center of a nilpotent Lie algebra $\mathfrak{g}$ is always nonempty as $[x,y]=0$ for all $x\in\mathfrak{g}^{p-1}, \ y\in\mathfrak{g},$ where $p$ is the index of nilpotency of $\mathfrak{g}.$

Let $\mathfrak{g}$ be a nilpotent Lie algebra and $\left\{e_{1}, \ldots, e_{k}\right\}$ be generators of $\mathfrak{g}.$ Then, due to a result of \cite{MZh96}, $\mathfrak{g}$ admits a basis $\mathcal{B}$ of the following form:
$$
\left\{e_{1}, \ldots, e_{k}, \left[e_{i_1}, [e_{i_2}, \cdots, [e_{i_{m-1}}, e_{i_m}], \cdots]\right]: 1\le i_1, \ldots, i_m \le k, 1\le m \le k\right\}.
$$

A linear transformation $d$ of a Lie algebra $\mathfrak{g}$ is {\it a derivation} if it satisfies Leibniz rule, that is, equality
$$d([x,y])=[d(x),y]+[x,d(y)]$$
holds for any $x, y\in \mathfrak{g}.$ The set of all derivations of $\mathfrak{g}$ is denoted by ${\rm Der}(\mathfrak{g}).$ Note that ${\rm Der}(\mathfrak{g})$ forms a Lie algebra under commutator. It is known that any linear transformation of $\mathfrak{g}$ sending every element to the element of center of $\mathfrak{g}$ is a derivation. Such derivations are called central derivations. Moreover, a linear transformation $d_0$ sending every element $x$ to zero element is called {\it trivial} derivation. Throughout this paper, when we require or claim the existence of a derivation, we always mean a non-trivial derivation.

Let $x\in\mathfrak{g}.$ A linear mapping defined as
\[{\rm ad}x(y)=[x,y]\]
is a derivation of $\mathfrak{g}.$ Such derivations are called inner derivations and the set of all inner derivations is denoted by ${\rm Inn}(\mathfrak{g}).$

{\it An automorphism} $\varphi$ of a Lie algebra $\mathfrak{g}$ is defined as a linear bijection $\varphi: \mathfrak{g}\to \mathfrak{g}$ such that
$$\varphi([x,y])=[\varphi(x),\varphi(y)]$$
for any $x, y\in \mathfrak{g}.$ The set of all automorphisms of $\mathfrak{g}$ forms a group, which is denoted by ${\rm Aut}(\mathfrak{g}).$ Note that, the identity map ${\rm Id}$ on $\mathfrak{g}$ is an automorphism. It is called a {\it trivial} automorphism. An automorphism $\varphi$ is called {\it unipotent}, if $\varphi-{\rm Id}$ is nilpotent.

Let $D$ be a derivation of a Lie algebra $\mathfrak{g}$ over $\mathbb{F}.$ If $char\mathbb{F}=0,$ then the exponential map defined as
\[{\rm exp}(D)=\sum_{k=0}^{\infty}\frac{D^k}{k!}\]
is an automorphism of $\mathfrak{g}.$ If $D$ is an inner derivation, then ${\rm exp}(D)$ is called inner automorphism.

A linear mapping  $\Delta$ is said to be {\it a local  derivation}, if for every $x\in \mathfrak{g}$ there exists a derivation $d_x$ on $\mathfrak{g}$ (depending on $x$) such that $\Delta(x)=d_x(x)$. Similarly, a linear mapping $\nabla$ is called {\it a local automorphism}, if for every $x\in \mathfrak{g}$ there exists automorphism $\varphi_x$ such that $\nabla(x)=\varphi_x(x)$. The set of local derivations (local automorphisms) of $\mathfrak{g}$ is denoted by ${\rm LocDer}(\mathfrak{g})$ (respectively ${\rm LocAut}(\mathfrak{g})$) Note that the set of all local automorphisms of an algebra form a multiplicative group, which contains the group of automorphisms as a subgroup (see \cite[Lemma 4]{E11}).

Let $f$ be a linear transformation on a Lie algebra $\mathfrak{g}$ and $A$ be a subspace of $\mathfrak{g}.$
$\mathcal{B}=\{e_1,e_2,\dots,e_n\}$ is a basis of $\mathfrak{g}$ such that $\mathcal{B_k}=\{e_1,e_2,\dots,e_k\}$ basis of $A$.
By the expression $f_{|A},$ we mean the linear mapping which is defined as follows:
\begin{align*}
f_{|A}(e_i)= \left\{\begin{array}{cccll}
                f(e_i),& \mbox{ if }& e_i\in A,\\
                0,& \mbox{ if }& e_i\in\mathfrak{g}\setminus A.
                \end{array}\right.
\end{align*}

\section{Local derivations of nilpotent Lie algebras}

Let $\mathfrak{n}$ be a finite-dimensional nilpotent Lie algebra over a field $\mathbb{F}$ with index of nilpotency $p,$ that is, $\mathfrak{n}^p=0.$ Let $\mathcal{B}=\{e_1,e_2,\dots,e_n\}$ be a basis of $\mathfrak{n}.$ At first, let us show that, if $\mathfrak{n}$ is $2$-step nilpotent Lie algebra, that is $p=3$, then it has a diagonal derivation. Indeed, consider a linear mapping ${\rm D}$ which is defined as
\begin{align}{\label{der1}}
\begin{array}{cccccccll}
{\rm D}_{\lambda}(e_i)=\lambda e_i,& & \mbox{if } &  e_i\in\mathfrak{n}\setminus\mathfrak{n}^2,\\
{\rm D}_{\lambda}(e_i)=2\lambda e_i,& & \mbox{if } &   e_i\in\mathfrak{n}^2.
\end{array}
\end{align}

We show that ${\rm D}_{\lambda}$ is a derivation. As expected, let $x,y\in\mathfrak{n}$ and $x=\sum\limits_{i}\alpha_ie_i, \ y=\sum\limits_{j}\beta_je_j.$ Then, we have
\begin{align*}
{\rm D}([x,y])&={\rm D}_{\lambda}\Big(\sum_{i,j}\alpha_i\beta_j[e_i,e_j]\Big)=\sum_{i,j}\alpha_i\beta_j{\rm D}_{\lambda}\left([e_i,e_j]\right)\\
                &=2\lambda\sum_{i,j}\alpha_i\beta_j[e_i,e_j]=\sum_{i,j}\alpha_i\beta_j[\lambda e_i,e_j]+\sum_{i,j}\alpha_i\beta_j[e_i,\lambda e_j]\\
                &=\sum_{i,j}\alpha_i\beta_j[{\rm D}_{\lambda}(e_i),e_j]+\sum_{i,j}\alpha_i\beta_j[e_i,{\rm D}_{\lambda}(e_j)]=[{\rm D}_{\lambda}(x),y]+[x,{\rm D}_{\lambda}(y)].
\end{align*}
This implies that ${\rm D}_{\lambda}$ is a derivation.

\begin{theorem}{\label{mainthm1}}
Let $\mathfrak{n}$ be a $2$-step nilpotent Lie algebra over a field $\mathbb{F}$ $(char\mathbb{F}\neq2).$ Then, $\mathfrak{n}$ admits a local derivation which is not derivation.

\end{theorem}

\begin{proof}
Let ${\rm D}$ be a derivation of $\mathfrak{n}$ which is defined as \eqref{der1} and $\mathcal{B}=\{e_1,e_2,\dots,e_m,\dots, e_n\}$ be a basis of $\mathfrak{n},$ where the number $m$ means the number of generators of $\mathfrak{n}.$ Consider a linear mapping $\Delta$ defined as
\begin{align*}
\Delta(e_i)= \left\{\begin{array}{cccll}
                0,& \mbox{ if }& e_i\in\mathfrak{n}\setminus\mathfrak{n}^2,\\
                2e_i,& \mbox{ if }& e_i\in\mathfrak{n}^2.
                \end{array}\right.
\end{align*}

We shall show that $\Delta$ is a local derivation which is not derivation. First, let us show that it is not derivation. Let $[e_r,e_t]=e_q$ for some $e_r,e_t\in\mathfrak{n}\setminus\mathfrak{n}^2, \ e_q\in\mathfrak{n}^2.$ Then, we have
\[[\Delta(e_r),e_t]+[e_r,\Delta(e_t)]=0\neq2e_q=2[e_r,e_t]=\Delta([e_r,e_t]).\]
Thus, $\Delta$ is not derivation.

Now, we show that it is a local derivation.

Let $x=\sum\limits_{i=1}^{n}\lambda_ie_i,$ then $\Delta(x)=2\sum\limits_{i=m+1}^{n}\lambda_ie_i.$
Consider the following possible cases:
\begin{itemize}

\item[1.] Let $\lambda_j\neq0$ for some $j$ $(1\leq j\leq m).$ Then, consider a derivation defined as
\begin{align*}
D_x(e_j)&=\frac{2}{\lambda_j}\sum\limits_{i=m+1}^{n}\lambda_ie_i,\\
D_x(e_i)&=0, \quad i=1,\dots,n, \quad i\neq j.
\end{align*}

Hence, we derive
$$D_x(x)=\sum\limits_{i=1}^{n}\lambda_iD_x(e_i)=\lambda_jD_x(e_j)=\lambda_j\frac{2}{\lambda_j}\sum\limits_{i=m+1}^{n}\lambda_ie_i=2\sum\limits_{i=m+1}^{n}\lambda_ie_i=\Delta(x).$$

\item[2.] Let $\alpha_j=0$ for all $j$ $(1\leq j\leq m).$ Then, we have $\Delta(x)=2x.$ Define a derivation $D_{x}$ as $D_{x}={\rm D}_{\bf 1},$ i.e.,
$$\begin{array}{cccccccll}
{\rm D}_x(e_i)=e_i,& & \mbox{if } &  e_i\in\mathfrak{n}\setminus\mathfrak{n}^2,\\
{\rm D}_x(e_i)=2 e_i,& & \mbox{if } &   e_i\in\mathfrak{n}^2.
\end{array}
$$

 Thus, we have $D_x(x)={\rm D}_{1}(x)=2x=\Delta(x).$
\end{itemize}
The proof is complete.
\end{proof}

To illustrate Theorem \ref{mainthm1}, consider the following example of $2$-step nilpotent Lie algebra:

\begin{example}
Let $\mathfrak{h}_n$ be complex $(2n+1)$-dimensional Heisenberg algebra with a basis $\mathcal{B}=\{e_{-n},\dots,e_{-1},e_0,e_1,\dots,e_n\}.$ The multiplication table is given as follows:
\begin{align}{\label{heisenberg}}
[e_{-i},e_{i}]=-[e_{i},e_{-i}]=e_{0},
\end{align}
where $i=1,\dots,n.$ The omitted products are supposed to be zero.

A linear mapping $\Delta$ defined as
\begin{align*}
\Delta(e_{0})=e_{0}, \qquad \Delta(e_i)=0, \quad i\neq0,
\end{align*}
is a local derivation by Theorem \ref{mainthm1}. Since
\[[\Delta(e_{-1}),e_{1}]+[e_{-1},\Delta(e_{1})]=0\neq e_{0}=\Delta(e_0)=\Delta([e_{-1},e_{1}]),\]
we conclude that $\Delta$ is not derivation.
\end{example}

\begin{remark}
It should be noted that, in Theorem \ref{mainthm1} the condition $char\mathbb{F}\neq2$ is essential. General behaviour of local derivations on nilpotent Lie algebras over the field of characteristic $2$ is different. There is a nilpotent Lie algebra over the field of characteristic $2$ for which the spaces of local derivations and derivations coincide.
\end{remark}

The following example shows that Theorem \ref{mainthm1} is not true over the field of characteristic $2:$

\begin{example}
Let $\mathfrak{s}$ be a $12$-dimensional nilpotent Lie algebra over the field $\mathbb{Z}_2$ with a basis $\mathcal{B}=\{e_1,e_2,\dots,e_{12}\}$ and the following table of multiplication:
\begin{align*}
[e_2,e_1]&=e_8, & [e_1,e_6]&=e_8, & [e_1,e_3]&=e_9, & [e_1,e_5]&=e_{9}, & [e_5,e_6]&=e_{9},\\
[e_1,e_4]&=e_{10}, & [e_2,e_3]&=e_{10}, & [e_1,e_7]&=e_{10}, &  [e_5,e_7]&=e_{10}, & [e_4,e_3]&=e_{11},\\
[e_7,e_6]&=e_{11}, & [e_2,e_7]&=e_{12}, & [e_3,e_6]&=e_{12},
\end{align*}
where the omitted products are supposed to be zero.

Direct computations show that every derivation ${\rm D}$ of $\mathfrak{s}$ acts on basis elements as follows:
\begin{align}{\label{der=lder2}}
{\rm D}(e_i)= \left\{\begin{array}{cccll}
                \alpha e_i+{\rm T}(e_i),& \mbox{ if }& 1\leq i\leq 7,\\
                0,& \mbox{ if }& 8\leq i\leq 12,
                \end{array}\right.
\end{align}
where $\alpha\in\mathbb{Z}_2$ and ${\rm T}$ is a central derivation.

Let $\Delta$ be a local derivation of $\mathfrak{s}.$ Then, $\Delta$ is defined as follows:
\begin{align*}
\Delta(e_i)= \left\{\begin{array}{cccll}
                \beta_i e_i+{\rm T_{e_i}}(e_i),& \mbox{ if }& 1\leq i\leq 7,\\
                0,& \mbox{ if }& 8\leq i\leq 12,
                \end{array}\right.
\end{align*}
where $\beta_i\in\mathbb{Z}_2$ and ${\rm T}_{e_{i}}$ are central derivations.

Now, we show that every local derivation $\Delta$ on $\mathfrak{s}$ is a derivation. Indeed, consider an element $x=\sum\limits_{i=1}^{7}e_i.$ Applying $\Delta(x)={\rm D}_x(x)$ we have
\begin{align*}
\Delta(x)=\sum_{i=1}^{7}\beta_ie_i+y=\alpha\sum_{i=1}^{7}e_i+z={\rm D}_{x}(x),
\end{align*}
where $y,z\in Z(\mathfrak{s}).$ By comparing the corresponding coefficients at basis elements we obtain $\beta_1=\beta_2=\dots=\beta_7.$ Hence, $\Delta$ is of the form \eqref{der=lder2}, that is, $\Delta$ is a derivation.

\end{example}

Now, we consider more general case. We define a necessary condition under which a nilpotent Lie algebra admits a pure local derivation.

\begin{theorem}{\label{mainthm2}}
Let $\mathfrak{n}$ be a nilpotent Lie algebra over a field $\mathbb{F}$ with nilindex $p$ $(p\geq4).$ If there exists a derivation ${\rm D}\in {\rm Der}({\mathfrak{n}}),$ such that ${\rm D}(\mathfrak{n}^2)\subseteq Z(\mathfrak{n}),$ then $\mathfrak{n}$ admits a local derivation which is not derivation.

\end{theorem}

\begin{proof}
Let $\mathcal{B}=\{e_1,e_2,\dots,e_m,\dots, e_n\}$ be a basis of $\mathfrak{n},$ where the number $m$ means the number of generator elements of $\mathfrak{n}.$ Suppose that, $[e_r,e_t]=e_q$ for some $e_r,e_t\in\mathfrak{n}\setminus\mathfrak{n}^2, \ e_q\in\mathfrak{n}^2\setminus\mathfrak{n}^3.$

If ${\rm D}(e_i)=0$ for all $e_i\in\mathfrak{n}^2\setminus\mathfrak{n}^3,$ then ${\rm D}$ would be a trivial derivation which contradicts the existence of derivation with ${\rm D}(\mathfrak{n}^2)\subseteq Z(\mathfrak{n}).$ Therefore, there exists $e_{i_0}\in\mathfrak{n}^2$ such that ${\rm D}(e_{i_0})\neq0.$

Without loss of generality we may assume that ${\rm D}(e_q)\neq0.$ Then, define a linear mapping $\Delta:\mathfrak{n}\to\mathfrak{n}$ as follows:
$$\Delta={\rm D}_{|\mathfrak{n}^2}.$$

We shall show that it is a local derivation which is not derivation. Since
$$[\Delta(e_r),e_t]+[e_r,\Delta(e_t)]=0\neq {\rm D}(e_q)=\Delta([e_r,e_t]),$$
we conclude that $\Delta$ is not derivation.

However, $\Delta$ is a local derivation. Indeed, let $x=\sum\limits_{i=1}^{n}\lambda_ie_i.$ Then, $\Delta(x)=y\in Z(\mathfrak{n}).$

Consider the following possible cases:
\begin{itemize}

\item Let $\lambda_j\neq0$ for some $j$ $(1\leq j\leq m).$ Then, define a derivation as follows:
\begin{align*}
D_x(e_j)&=\frac{1}{\lambda_j}y,\\
D_x(e_i)&=0, \quad i=1,\dots,n, \quad i\neq j.
\end{align*}

Hence, we obtain that
$$D_x(x)=\lambda_jD_x(e_j)=\lambda_j\frac{1}{\lambda_j}y=y=\Delta(x).$$

\item Let $\lambda_j=0$ for all $j$ $(1\leq j\leq m).$ Then, consider a derivation $D_x$ defined as
$$D_{x}={\rm D}.$$
Hence, we derive
$$D_x(x)={\rm D}(x)={\rm D}\left(\sum\limits_{i=m+1}^{n}\lambda_ie_i\right)=y=\Delta(x).$$

\end{itemize}
The proof is complete.
\end{proof}

Now, we state some implications of Theorem \ref{mainthm2}.

\begin{proposition}
Let $\mathfrak{n}$ be a nilpotent Lie algebra with nilindex $4.$ Then, $\mathfrak{n}$ admits a local derivation which is not derivation.
\end{proposition}

\begin{proof}
According to the result above, it is sufficient to show that there exists a derivation $D$ which sends $\mathfrak{n}$ to the center. Since $\mathfrak{n}$ is an algebra of nilindex $4,$ we have $\mathfrak{n}^4=0$ and $0\neq\mathfrak{n}^3\subseteq Z(\mathfrak{n}).$ Thus, there exist $x\in\mathfrak{n}, \ y\in\mathfrak{n}^{2}$ such that $[x,y]\neq0.$ Consider an inner derivation $ad_x.$ Having $[x,y]\in[\mathfrak{n},\mathfrak{n}^{2}]\subseteq\mathfrak{n}^{3}\subseteq Z(\mathfrak{n})$ in mind, we have a derivation $D=ad_x$ which satisfies with the condition of Theorem \ref{mainthm2}.
\end{proof}

\begin{corollary}{\label{cor0}}
Let $\mathfrak{n}$ be a nilpotent Lie algebra with $\mathfrak{n}^{p}=0$ $(p\geq5)$. If $[\mathfrak{n}^{p-3},\mathfrak{n}^{2}]\neq0,$ then $\mathfrak{n}$ admits a local derivation which is not derivation.
\end{corollary}

\begin{proof}Let $\mathfrak{n}$ be a nilpotent Lie algebra of nilindex $p$ $(p\geq5).$ The condition $[\mathfrak{n}^{p-3},\mathfrak{n}^{2}]\neq0$ guarantees the existence of $x\in\mathfrak{n}^{p-3}, \ y\in\mathfrak{n}^{2}$ such that $[x,y]\neq0.$ Consider an inner derivation $ad_x.$ Since $[x,y]\in[\mathfrak{n}^{p-3},\mathfrak{n}^{2}]\subseteq\mathfrak{n}^{p-1}\subseteq Z(\mathfrak{n}),$ we have a derivation $d=ad_x$ which satisfies with the condition of Theorem \ref{mainthm2}.
\end{proof}

Here, for this kind of algebras there is an inner derivation providing the existence of pure local derivation. However, there are many nilpotent Lie algebras for which $[\mathfrak{n}^{p-3},\mathfrak{n}^{2}]=0,$ so we can not construct a local derivation by using inner derivations. Interestingly, for these algebras there exists outer derivations fulfilling the condition of Theorem \ref{mainthm2}.

For instance, consider the following example:
\begin{example}{\label{exam2}} Let $n\geq3$ and $\mathfrak{w}$ be a complex $n$-dimensional Witt algebra with basis  $\mathcal{B}=\{e_1,e_2,\dots,e_n\}$ and the following basis products:
\[[e_i,e_j]=(j-i)e_{i+j}, \quad 1\leq i,j\leq n, \ i+j\leq n.\]
Note that $\mathfrak{w}^k=span\{e_{k+1},\dots,e_n\}$ and $Z(\mathfrak{w})=\{e_n\}.$ Thus, $\mathfrak{w}^n=0$ which implies $[\mathfrak{w}^{n-3},\mathfrak{w}^{2}]=0.$ It means that we can not construct an inner derivation satisfying the condition of Theorem \ref{mainthm2}. However, a linear mapping $d$ defined as
\begin{align*}
d(e_2)=e_{n-1}, && d(e_3)=(n-2)e_n, && d(e_i)=0, \quad i\neq2,3,
\end{align*}
is a derivation of $\mathfrak{w}.$ Moreover, this derivation sends $e_3$ to the center of $\mathfrak{w}.$ Thus, by Theorem \ref{mainthm2}, a linear transformation $\Delta$ defined as
\begin{align*}
\Delta(e_3)=(n-2)e_n, &\quad \Delta(e_i)=0, \quad i\neq3,
\end{align*}
is local derivation which is not derivation.
\end{example}

Now, we generalize Example \ref{exam2}.
\begin{proposition}{\label{prop2gen}}
Let $\mathfrak{n}$ be a two-generated nilpotent Lie algebra. Then, $\mathfrak{n}$ admits a local derivation which is not derivation.
\end{proposition}

\begin{proof}
At first, we construct a derivation which sends $\mathfrak{n}^2$ to the center of the algebra $Z(\mathfrak{n})$. Let $\mathcal{B}=\{e_1,e_2,\dots,e_n\}$ be a basis and $e_1,e_2$ be generator elements of $\mathfrak{n}.$ Suppose that, $\mathfrak{n}^{p}=0$ for some $p.$ Without loss of generalization, we assume that $[e_1,e_2]=e_3.$ Since $\mathfrak{n}^{p-1}\neq0,$ there exists $x\in\mathfrak{n}^{p-2}$ such that either $[e_1,x]\neq0$ or $[e_2,x]\neq0.$ Let $[e_1,x]=y\in\mathfrak{n}^{p-1}.$ Now, we show that a linear mapping $\delta$ defined as
\begin{align*}
\delta(e_2)=x, \quad  \delta(e_3)=y,\quad\quad
 \delta(e_i)=0, \quad  i\neq2,3
\end{align*}
is a derivation of $\mathfrak{n}.$ Indeed, we have
\begin{align*}
\delta([e_1,e_2])&=\delta(e_3)=y=[e_1,x]=[\delta(e_1),e_2]+[e_1,\delta(e_2)],\\
\delta([e_1,e_i])&=\delta\big(\sum_{j\geq4} (\ast)e_j\big)=0=[\delta(e_1),e_i]+[e_1,\delta(e_i)], \quad i\geq3,\\
\delta([e_2,e_i])&=\delta\big(\sum_{j\geq4} (\ast)e_j\big)=0=[\delta(e_2),e_i]+[e_2,\delta(e_i)], \quad i\geq3,\\
 \delta([e_i,e_j])&=0=[\delta(e_i),e_j]+[e_i,\delta(e_j)], &&& i,j\neq2,3.
\end{align*}

Hence, $\delta$ is a derivation. Moreover, $\delta$ fulfills the condition $\delta(\mathfrak{n}^2)\subseteq Z(\mathfrak{n}).$ Thus, by Theorem \ref{mainthm2}, $\mathfrak{n}$ admits a local derivation which is not derivation. The case $[e_1,x]=0, \ [e_2,x]\neq0$ can be proved similarly.
\end{proof}

In fact, a Lie algebra can not be generated by one element. Generally, we do not know the situation in other non-associative algebras. There exists nilpotent algebras for which every local derivation is derivation. For instance, consider the following example:

\begin{example}
Let $\mathfrak{a}$ be a complex $n$-dimensional algebra with a basis $\mathcal{B}=\{e_1,e_2,\dots,e_n\}$ $(n\geq3)$ and multiplication table
\begin{align}{\label{1gen}}
e_ie_i=e_{i+1}, \quad 1\leq i\leq n-1.
\end{align}

Direct computations show that every derivation $D$ of $\mathfrak{a}$ acts on basis elements as follows:
\begin{align*}
D(e_1)=\alpha e_1+\beta e_n,\qquad
D(e_i)=2^{i-1}\alpha e_i,
\end{align*}
where $\alpha,\beta\in\mathbb{C},$ $2\leq i\leq n.$

We shall show that every local derivation on $\mathfrak{a}$ is a derivation. From the definition of local derivation, we derive that any local derivation $\Delta$ acts on basis elements as
\begin{align*}
\Delta(e_1)=\delta e_1+\zeta e_n, \qquad
\Delta(e_i)=\xi_i e_i,
\end{align*}
where $\delta,\zeta,\xi_i\in\mathbb{C},$ $2\leq i\leq n.$

To prove that ${\rm LocDer}(\mathfrak{a})={\rm Der}(\mathfrak{a}),$ we need to show that $\xi_i=2^{i-1}\delta$ for any $i=1,\dots n.$

Take an element $x=e_1+e_{i}, \ (2\leq i\leq n-1).$ Then, there exists a derivation
$D_x$ such that $D_x(x)=\Delta_x(x).$ Hence, we derive
\begin{align*}
D_x(x)&=\alpha_x e_1+2^{i-1}\alpha_x e_i+\beta_xe_n=\delta e_1+\xi_ie_i+\zeta e_n=\Delta(x).
\end{align*}
Comparing the corresponding coefficients we have $\xi_i=2^{i-1}\delta.$ Thus, it only remains to show that $\xi_n=2^{n-1}\delta.$ By taking an element $x=e_{n-1}+e_n$ and applying $D_x(x)=\Delta_x(x)$ we have
$$
D_x(x)=2^{n-2}\alpha_x e_{n-1}+2^{n-1}\alpha_x e_n=2^{n-2}\delta e_{n-1}+\xi_ne_n=\Delta(x).
$$
Comparing respective coefficients, we obtain the required relation. So, $\Delta$ is a derivation.

\end{example}

The example above is one generated algebra. In general, there exists a two-generated algebra for which every local derivation is a derivation.
\begin{example}
Let $\mathfrak{c}$ be a complex $6$-dimensional two-generated commutative algebra with a basis $\mathcal{B}=\{e_1,e_2,\dots,e_6\}$ $(n\geq3)$ and multiplication table
\begin{align*}
e_1e_2&=e_{3}, & e_1e_3&=e_{4}, & e_1e_4&=e_{5}, & e_2e_3&=e_{5}, & e_2e_4&=e_{6},
\end{align*}

Direct computations show that every derivation $D$ of $\mathfrak{c}$ has the following form:
\[
\begin{pmatrix}
\alpha & 0 & 0 & 0 & 0 & 0 \\

0 & 2\alpha & 0 & 0 & 0 & 0 \\

0 & 0 & 3\alpha & 0 & 0 & 0 \\

0 & 0 & 0 & 4\alpha & 0 & 0 \\

a_{51} & a_{52} & 0 & 0 & 5\alpha & 0 \\

a_{61} & a_{62} & 0 & 0 & 0 & 6\alpha \\

\end{pmatrix}.
\]

Let $\Delta$ be a local derivation of $\mathfrak{c}.$ Then, it has the following matrix form:
\[
\begin{pmatrix}
\mu_1 & 0 & 0 & 0 & 0 & 0 \\

0 & \mu_2 & 0 & 0 & 0 & 0 \\

0 & 0 & \mu_3 & 0 & 0 & 0 \\

0 & 0 & 0 & \mu_4 & 0 & 0 \\

b_{51} & b_{52} & 0 & 0 & \mu_5 & 0 \\

b_{61} & b_{62} & 0 & 0 & 0 & \mu_6 \\

\end{pmatrix}.
\]

Consider the elements $e_1+e_2+e_3+e_4\in\mathfrak{c}.$ By definition, we have $\Delta(e_1+e_2+e_3+e_4)=D(e_1+e_2+e_3+e_4).$ Hence, we have
\begin{align*}
\Delta(e_1+e_2+e_3+e_4)&=\alpha e_1+2\alpha e_2+3\alpha e_3+4\alpha e_4+(a_{51}+a_{52})e_5+(a_{61}+a_{62})e_6\\
&=\mu_1 e_1+\mu_2 e_2+\mu_3 e_3+\mu_4 e_4+(b_{51}+b_{52})e_5+(b_{61}+b_{62})e_6\\
&=D(e_1+e_2+e_3+e_4).
\end{align*}
Comparing the corresponding coefficients, we have $\mu_2=2\mu_1, \ \mu_3=3\mu_1, \ \mu_4=4\mu_1.$ Applying the same argument for the element $e_4+e_5+e_6,$ we obtain $\mu_5=5\mu_1, \ \mu_6=6\mu_1.$ Thus, $\Delta$ is a derivation.

\end{example}


\section{Local automorphisms of nilpotent Lie algebras}

In this part, we will show that any $2$-step nilpotent Lie algebra admits a local automorphism which is not derivation. Moreover, a nilpotent Lie algebra admits pure local automorphism under the condition of Theorem \ref{mainthm2}. At first, let us show that there exists a non unipotent automorphism on a $2$-step nilpotent Lie algebra over the field of characteristic zero.

Let $\mathfrak{n}$ be a $2$-step nilpotent Lie algebra with a basis $\mathcal{B}=\{e_1,e_2,\dots,e_n\}.$ Consider a linear mapping $\psi$ defined as
\begin{align*}
\psi(e_i)&=\varepsilon e_i, \quad \mbox{if }   \quad e_i\in\mathfrak{n}\setminus\mathfrak{n}^2,\\
\psi(e_i)&=\varepsilon^{2} e_i,  \quad \mbox{if }    \quad e_i\in\mathfrak{n}^2,
\end{align*}
where $\varepsilon\in\mathbb{F}\setminus\{0\}.$ Then, $\psi$ is an automorphism of $\mathfrak{n}.$ Indeed, let $x,y\in\mathfrak{n}$ and $x=\sum\limits_{i}\alpha_ie_i, \ y=\sum\limits_{j}\beta_je_j.$ Then, we have
\begin{align*}
\psi([x,y])&=\psi\Big(\sum_{i,j}\alpha_i\beta_j[e_i,e_j]\Big)=\sum_{i,j}\alpha_i\beta_j\psi\left([e_i,e_j]\right)\\
                &=\varepsilon^2\sum_{i,j}\alpha_i\beta_j[e_i,e_j]=\sum_{i,j}\alpha_i\beta_j[\varepsilon e_i,\varepsilon e_j]\\
                &=\sum_{i,j}\alpha_i\beta_j[\psi(e_i),\psi(e_j)]=[\psi(x),\psi(y)].
\end{align*}
This implies that $\psi$ is an automorphism.

Consider the derivation ${\rm D}_{\lambda}$ which is defined as \eqref{der1}. Then, ${\rm exp}({\rm D}_{\lambda})$ is an automorphism that acts on basis elements as follows:
\begin{align}{\label{aut1}}
\begin{array}{cccccccll}
{\rm exp}({\rm D}_{\lambda})(e_i)={\bf e}^{\lambda}e_i,& & \mbox{if } &  e_i\in\mathfrak{n}\setminus\mathfrak{n}^2,\\
{\rm exp}({\rm D}_{\lambda})(e_i)={\bf e}^{2\lambda} e_i,& & \mbox{if } &   e_i\in\mathfrak{n}^2,
\end{array}
\end{align}
where ${\bf e}$ is the exponential constant.

\begin{theorem}{\label{mainthm3}}
Let $\mathfrak{n}$ be a $2$-step nilpotent Lie algebra over a field $\mathbb{F}$ $(char\mathbb{F}\neq2).$ Then, $\mathfrak{n}$ admits a local automorphism which is not automorphism.

\end{theorem}

\begin{proof}
Let ${\rm D}_{\bf 1}$ be a derivation defined as \eqref{der1} and $\Phi={\rm exp}({\rm D}_{\bf 1})$ be corresponding automorphism as \eqref{aut1}. Consider a linear mapping $\nabla$ defined as follows:
\begin{align*}
\nabla(x)&={\rm Id}(x), \ \mbox{ if } x\in\mathfrak{n}\setminus\mathfrak{n}^2,\\
\nabla(x)&=\Phi(x), \ \mbox{ if } x\in\mathfrak{n}^2,
\end{align*}
where ${\rm Id}$ is the identity map.

We shall show that $\nabla$ is a local automorphism which is not an automorphism. First, let us prove that it is not an automorphism. Since $\mathfrak{n}$ is nonabelian, there exist $e_r,e_t\in\mathfrak{n}\setminus\mathfrak{n}^2, \ e_k\in\mathfrak{n}^2$ such that
$[e_r,e_t]=e_k.$ Hence, we derive
\[[\nabla(e_r),\nabla(e_t)]=[e_r,e_t]=e_k\neq{\bf e}^2e_k=\nabla(e_k)=\nabla([e_r,e_t]),\]
where ${\bf e}$ is the exponential constant. Thus, $\nabla$ is not automorphism.

Now, we show that $\nabla$ is a local automorphism. Let $\mathcal{B}=\{e_1,e_2,\dots,e_m,\dots, e_n\}$ be a basis of $\mathfrak{n},$ where the number $m$ means the number of generators of $\mathfrak{n}.$ For an element $x=\sum\limits_{i=1}^{n}\lambda_ie_i,$ we have
\[\nabla(x)=\sum_{i=1}^{m}\lambda_ie_i+{\bf e}^2\sum_{i=m+1}^{n}\lambda_ie_i.\]

 Consider the following possible cases:
 \begin{itemize}
 \item[1.] If $\lambda_j\neq0$ for some $j, \ (1\leq j\leq m),$ then we choose a derivation $D_x$ that acts on basis elements as follows:
 \begin{align*}
 D_x(e_j)&=\frac{{\bf e}^2-1}{\lambda_j}\sum_{i=m+1}^{n}\lambda_ie_i,\\
 D_x(e_r)&=0,
 \end{align*}
 for all $r\in\{1,2,\dots,m\}, \ r\neq j.$
 Using this derivation, we construct an automorphism $\Phi_x={\rm exp}(D_x).$ Having in mind $D_x^2=0,$ we have
 \[{\rm exp}(D_x)={\rm Id}+D_x.\]
 Hence, we derive
 \begin{align*}
 {\rm exp}(D_x)(x)&={\rm Id}(x)+D_x(x)=x+\sum_{i=1}^{n}\lambda_iD_x(e_i)\\
 &=x+\lambda_jD_x(e_j)=\sum\limits_{i=1}^{n}\lambda_ie_i+\lambda_j\frac{{\bf e}^2-1}{\lambda_j}\sum_{i=m+1}^{n}\lambda_ie_i\\
 &=\sum_{i=1}^{m}\lambda_ie_i+{\bf e}^2\sum_{i=m+1}^{n}\lambda_ie_i=\nabla(x).
 \end{align*}

 \item[2.] If $\lambda_j\neq0$ for all $j$ $(1\leq j\leq m),$ then we choose a derivation ${\rm D}_{\bf 1}$ and construct an automorphism $\Phi={\rm exp}({\rm D}_{\bf 1}).$ Choosing $\Phi_x=\Phi$ and taking ${\rm D}_{\bf 1}^2=0$ into account, we derive
    \[\Phi_x(x)={\rm exp}({\rm D}_{\bf 1})(x)=x+{\rm D}_{\bf 1}(x)=\Phi(x)=\nabla(x).\]

 \end{itemize}

Thus, $\nabla$ is a local automorphism. The proof is complete.
\end{proof}

To illustrate Theorem \ref{mainthm3}, consider the following example:

\begin{example}
Let $\mathfrak{h}_n$ be complex $(2n+1)$-dimensional Heisenberg algebra with a basis $\mathcal{B}=\{e_{-n},\dots,e_{-1},e_0,e_1,\dots,e_n\}$ and multiplication table as \eqref{heisenberg}. Consider a linear mapping $\nabla$ defined as
\begin{align*}
\nabla(e_i)=e_i, \quad i\neq0,\qquad
\nabla(e_{0})=2e_{0}.
\end{align*}

By Theorem \ref{mainthm3}, $\nabla$ is a local automorphism. Since
\[[\nabla(e_1),\nabla(e_{-1})]=[e_1,e_{-1}]=e_0\neq2e_0=\nabla(e_0)=\nabla([e_1,e_{-1}]),\]
we conclude that $\nabla$ is not an automorphism.




\end{example}

Now, let us give a sufficient condition under which a nilpotent Lie algebra admits pure local automorphism. Interestingly, the conditions under which a nilpotent Lie algebra admitting pure local derivation coincide with the conditions of admitting pure local automorphism.

\begin{theorem}{\label{mainthm4}}
Let $\mathfrak{n}$ be a nilpotent Lie algebra with nilindex $p$ $(p\geq4).$ If there exists a derivation ${\rm D}\in {\rm Der}({\mathfrak{n}})$ such that, ${\rm D}(\mathfrak{n}^2)\subseteq Z(\mathfrak{n}),$ then $\mathfrak{n}$ admits a local automorphism which is not automorphism.
\end{theorem}

\begin{proof}
Let ${\rm D}$ be derivation such that ${\rm D}(\mathfrak{n}^2)\subseteq Z(\mathfrak{n})$ and $\Phi={\rm exp}({\rm D})$ be an automorphism of $\mathfrak{n}.$ Let $\mathcal{B}=\{e_1,e_2,\dots,e_m,\dots, e_n\}$ be a basis of $\mathfrak{n},$ where the number $m$ means the number of generators of $\mathfrak{n}.$ Consider a linear mapping $\nabla$ defined as follows:
\begin{align*}
\nabla(x)&={\rm Id}(x), \ \mbox{ if } x\in\mathfrak{n}\setminus\mathfrak{n}^2,\\
\nabla(x)&=\Phi(x), \ \mbox{ if } x\in\mathfrak{n}^2,
\end{align*}
where ${\rm Id}$ is the identity map.

We shall prove that $\nabla$ is a local automorphism which is not an automorphism. First, we show that it is not an automorphism. Let
\[[e_r,e_t]=e_k,\]
where $e_r,e_t\in\mathfrak{n}\setminus\mathfrak{n}^2, \ e_k\in\mathfrak{n}^2\setminus\mathfrak{n}^3.$
 Without loss of generality, we may assume that ${\rm D}(e_k)\neq0.$
 Hence, having in mind ${\rm D}^2=0,$ we derive
\[[\nabla(e_r),\nabla(e_t)]=[e_r,e_t]=e_k\neq e_k+{\rm D}(e_k)={\rm exp}({\rm D})(e_k)=\nabla(e_k)=\nabla([e_r,e_t]).\]
 Thus, $\nabla$ is not automorphism.

Now, we show that it is a local automorphism. For an element $x=\sum\limits_{i=1}^{n}\lambda_ie_i,$ we have
\[\nabla(x)=\sum_{i=1}^{n}\lambda_ie_i+\sum_{i=m+1}^{n}\lambda_i{\rm D}(e_i).\]

 Consider the following possible cases:
 \begin{itemize}
 \item[1.] If $\lambda_j\neq0$ for some $j$ $(1\leq j\leq m),$ then we choose a derivation $D_x$ that acts on basis elements as follows:
 \begin{align*}
 D_x(e_j)&=\frac{1}{\lambda_j}\sum_{i=m+1}^{n}\lambda_i{\rm D}(e_i),\\
 D_x(e_r)&=0,
 \end{align*}
 for all $r\in\{1,2,\dots,m\},$ $r\neq j.$
 Using this derivation, we construct an automorphism $\Phi_x={\rm exp}(D_x).$ Taking $D_x^2=0$ into account, we have
 \[{\rm exp}(D_x)={\rm Id}+D_x.\]
 Hence, we derive
 \begin{align*}
 \Phi_x(x)&={\rm Id}(x)+D_x(x)=x+\sum_{i=1}^{n}\lambda_iD_x(e_i)\\
 &=x+\lambda_jD_x(e_j)=\sum\limits_{i=1}^{n}\lambda_ie_i+\lambda_j\frac{1}{\lambda_j}\sum_{i=m+1}^{n}\lambda_i{\rm D}(e_i)\\
 &=\sum_{i=1}^{n}\lambda_ie_i+\sum_{i=m+1}^{n}\lambda_i{\rm D}(e_i)=\nabla(x).
 \end{align*}

 \item[2.] If $\lambda_j\neq0$ for all $j$ $(1\leq j\leq m),$ then we choose a derivation ${\rm D}$ and construct an automorphism $\Phi={\rm exp}({\rm D}).$ Choosing $\Phi_x=\Phi$ and taking ${\rm D}^2=0$ and $x\in\mathfrak{n}^2$ into account, we derive
    \[\Phi_x(x)={\rm exp}({\rm D})(x)=x+{\rm D}(x)=\Phi(x)=\nabla(x).\]

 \end{itemize}

Thus, $\nabla$ is a local automorphism. The proof is complete.
\end{proof}

Now, we state some results which come from Theorem \ref{mainthm4}.

\begin{corollary}{\label{cor1}}
Let $\mathfrak{n}$ be a nilpotent Lie algebra with $\mathfrak{n}^{p}=0$. If $[\mathfrak{n}^{p-3},\mathfrak{n}^{2}]\neq0,$ then $\mathfrak{n}$ admits a local automorphism which is not automorphism.
\end{corollary}

\begin{proof}
The condition $[\mathfrak{n}^{p-3},\mathfrak{n}^{2}]\neq0$ yields the existence of elements $x\in\mathfrak{n}^{p-3}, \ y\in\mathfrak{n}^{2}$ such that $[x,y]\neq0.$ Consider an inner derivation $ad_x.$ Since $[x,y]\in[\mathfrak{n}^{p-3},\mathfrak{n}^{2}]\subseteq\mathfrak{n}^{p-1}\subseteq Z(\mathfrak{n}),$ we have a derivation $D=ad_x$ which satisfies with the condition of Theorem \ref{mainthm4}.
\end{proof}

\begin{corollary}
Let $\mathfrak{n}$ be a nilpotent Lie algebra with nilindex $4$. Then, $\mathfrak{n}$ admits a local automorphism which is not automorphism.
\end{corollary}

\begin{proof}
Since the nilindex of $\mathfrak{n}$ is $4,$ we have $[\mathfrak{n},\mathfrak{n}^2]=\mathfrak{n}^3\neq0.$ So, by Corollary \ref{cor1}, $\mathfrak{n}$ admits a pure local derivation.
\end{proof}

\begin{corollary}
Let $\mathfrak{n}$ be a two-generated nilpotent Lie algebra. Then, $\mathfrak{n}$ admits a local automorphism which is not automorphism.
\end{corollary}

\begin{proof}
In the proof of Proposition \ref{prop2gen}, it is shown that there exists a derivation satisfying the condition of Theorem \ref{mainthm4}.
\end{proof}

As was in local derivations, the general situation of local automorphisms might be different in one generated algebras. There exists nilpotent algebras for which every local automorphism is automorphism. For instance, consider the following example:

\begin{example}
Let $\mathfrak{a}$ be a complex $n$-dimensional algebra with a basis $\mathcal{B}=\{e_1,e_2,\dots,e_n\}$ $(n\geq3)$ and multiplication table is defined as \eqref{1gen}. Then, direct computations show that every automorphism $\Phi$ of $\mathfrak{a}$ acts on basis elements as follows:
\begin{align*}
\Phi(e_1)&=\alpha e_1+\beta e_n,\\
\Phi(e_i)&=\alpha^{2^{i-1}} e_i,
\end{align*}
where $\alpha,\beta\in\mathbb{C}, \ 2\leq i\leq n.$\\
We shall show that ${\rm LocAut(\mathfrak{a})}={\rm Aut}(\mathfrak{a}).$ From the definition of local automorphism, we derive that any local automorphism $\nabla$ acts on basis elements as
\begin{align*}
\nabla(e_1)&=\delta e_1+\zeta e_n,\\
\nabla(e_i)&=\xi_i e_i,
\end{align*}
where $\delta,\zeta,\xi_i\in\mathbb{C}, \ 2\leq i\leq n.$

It is enough to show that $\xi_i=\delta^{2^{i-1}}.$ Take an element $x=e_1+e_{i}, \ (2\leq i\leq n-1).$ Then, there exists a automorphism
$\Phi_x$ such that $\Phi_x(x)=\nabla_x(x).$ Hence, we derive
\begin{align*}
\Phi_x(x)&=\alpha_x e_1+\alpha_x^{2^{i-1}} e_i+\beta_xe_n=\delta e_1+\xi_ie_i+\zeta e_n=\nabla(x).
\end{align*}
Comparing the corresponding coefficients we have $\xi_i=\delta^{2^{i-1}}.$ Thus, it only remains to show that $\xi_n=\delta^{2^{n-1}}.$ By taking an element $x=e_{n-1}+e_n$ and applying $\Phi_x(x)=\nabla_x(x)$ we have
\begin{align*}
\Phi_x(x)&=\alpha_x^{2^{n-2}} e_{n-1}+\alpha_x^{2^{n-1}}e_n=\delta^{2^{n-2}} e_{n-1}+\xi_ne_n=\nabla(x).
\end{align*}
Comparing respective coefficients, we obtain the required relation. So, $\nabla$ is an automorphism.

\end{example}


\begin{thebibliography}{KMRT98}
	
\bibitem{AK16} Sh.A. Ayupov, K.K. Kudaybergenov, {\it Local derivations on finite-dimensional Lie algebras,} Linear Algebra Appl., {\bf 493} (2016), 381-398.

\bibitem{AK18}  Sh.A. Ayupov, K.K. Kudaybergenov, {\it Local automorphisms on finite-dimensional Lie and Leibniz algebras},  Algebra, complex analysis and plupotential theory, 31-44, Springer Proc. Math. Stat., 264, Springer,  2018.
	

\bibitem{AKO19} Sh.A. Ayupov, K.K. Kudaybergenov, B.A. Omirov, {\it Local and 2-local derivations and automorphisms on simple Leibniz algebras}, Bull. Malays. Math. Sci. Soc., {\bf 43} (3) (2020), 2199-2234.
	
\bibitem{AEK21} Sh.A. Ayupov, A. Elduque, K.K. Kudaybergenov, {\it Local and 2-local derivations of Cayley algebras}, J. Pure. Appl. Algebra, {\bf 227} (5) (2023),  107277	
	
\bibitem{AKh21}  Sh.A. Ayupov,  A.Kh. Khudoyberdiyev, {\it Local derivations on solvable Lie algebras}, Linear. Mult. Algebra, {\bf 69} (7) (2021), 1286-1301.	


\bibitem{BS93} M. Bresar, P.  Semrl, {\it Mappings which preserve idempotents, local automorphisms, and local derivations}, Canad. J. Math., {\bf 45} (3) (1993), 483-496.
	
\bibitem{CZZ21} Y. Chen, K. Zhao, Y. Zhao, {\it Local derivations on Witt algebras}, Linear. Mult. Algebra,
{\bf 70} (6) (2022),  1159-1172.	


\bibitem{CZZ2020} Y. Chen, K. Zhao, Y. Zhao, {\it Local and 2-local automorphisms of simple generalized Witt algebras}, Ark. Mat., {\bf 59} (1) (2021), 1-10.


\bibitem{Cons} M. Constantini, {\it Local automorphisms of finite-dimensional simple Lie algebras},  Linear Algebra Appl., \textbf{562} (1) (2019), 123-134.


\bibitem{E11}  A.P. Elisova, I.N. Zotov, V.M. Levchuk,  G.S. Suleymanova, {\it Local automorphisms and local derivations of nilpotent matrix algebras}, Izv. Irkutsk Gos. Univ., {\bf 4} (1) (2011), 9-19.


\bibitem{Elisova13} A.P. Elisova, {\it Local automorphisms of nilpotent algebras of matrices of small orders}, Russian Math., (Iz. VUZ) {\bf 57} (2) (2013),  34-41.	

\bibitem{FKK2021} B. Ferreira, I. Kaygorodov, K.K. Kudaybergenov, {\it Local and 2-local derivations of simple $n$-ary algebras}, Ricerche di Matematica, {\bf 73} (1) (2024),  341-350.

\bibitem{MZh96} D.J. Meng, L.Sh. Zhu, {\it Solvable complete Lie algebras. I.} Commun. Algebra, {\bf 24} (13) (1996), 4181-4197.

\bibitem{Kadison90} R.V. Kadison, {\it Local derivations,} J. Algebra,  \textbf{130} (1990). 494-509.
	
\bibitem{KKO21} K.K. Kudaybergenov, T.K. Kurbanbaev, B.A. Omirov, {\it Local derivations on solvable Lie algebras of maximal rank}, Commun. Algebra, \textbf{50} (9) (2022), 3816-3826.

\bibitem{KKO23} K.K. Kudaybergenov, T.K. Kurbanbaev, B.A. Omirov, {\it Local automorphisms of complex solvable Lie algebras of maximal rank}, Linear. Mult. Algebra, (2023), doi.org/10.1080/03081087.2023.2241610.
    	
\bibitem{Larson90} D.R. Larson, A.R. Sourour, {\it Local derivations and local automorphisms of \(B(X)\)}, Proc. Sympos. Pure Math., \textbf{51} (1990), 187-194.


\bibitem{WGL24} Q. Wu, Sh. Gao, D. Liu, {\it Local derivations on the Lie algebra $W(2,2)$}, Linear. Mult. ALgebra, {\bf 72} (2024), 631-343.

\bibitem{YC20}	Y. Yu, Zh. Chen, {\it Local derivations on Borel subalgebras of finite-dimensional simple Lie algebras}, Commun. Algebra, {\bf 48} (1) (2020),  1-10.
\end{thebibliography}
\end{document}